\documentclass{spie_fickus}

\usepackage{amsthm}
\usepackage{amssymb}
\usepackage{amsmath}
\usepackage{graphicx}

\newcommand{\bbC}{\mathbb{C}}
\newcommand{\bbF}{\mathbb{F}}

\newcommand{\bbR}{\mathbb{R}}

\newcommand{\bfA}{\mathbf{A}}
\newcommand{\bfB}{\mathbf{B}}

\newcommand{\bfD}{\mathbf{D}}

\newcommand{\bfG}{\mathbf{G}}

\newcommand{\bfI}{\mathbf{I}}
\newcommand{\bfJ}{\mathbf{J}}

\newcommand{\bfU}{\mathbf{U}}

\newcommand{\bfx}{\mathbf{x}}
\newcommand{\bfy}{\mathbf{y}}

\newcommand{\bfzero}{\boldsymbol{0}}
\newcommand{\bfone}{\boldsymbol{1}}

\newcommand{\bfphi}{\boldsymbol{\varphi}}

\newcommand{\bfPhi}{\boldsymbol{\Phi}}

\newcommand{\Tr}{\operatorname{Tr}}
\newcommand{\SRG}{\operatorname{SRG}}

\newcommand{\Fro}{\mathrm{Fro}}

\newcommand{\abs}[1]{|{#1}|}

\newcommand{\bigparen}[1]{\bigl({#1}\bigr)}

\newcommand{\bigbracket}[1]{\bigl[{#1}\bigr]}

\newcommand{\biggbracket}[1]{\biggl[{#1}\biggr]}

\newcommand{\set}[1]{\{{#1}\}}
\newcommand{\bigset}[1]{\bigl\{{#1}\bigr\}}

\newcommand{\norm}[1]{\|{#1}\|}

\newcommand{\ip}[2]{\langle{#1},{#2}\rangle}

\newtheorem{theorem}{Theorem}[section]

\newtheorem{lemma}[theorem]{Lemma}

\theoremstyle{definition}

\title{Detailing the equivalence between real equiangular tight frames and certain strongly regular graphs}

\author{Matthew Fickus and Cody E.\ Watson
\skiplinehalf
Department of Mathematics and Statistics, Air Force Institute of Technology\\Wright-Patterson Air Force Base, Ohio 45433, USA}
\authorinfo{Send correspondence to Matthew Fickus: E-mail: Matthew.Fickus@gmail.com}

\begin{document}
\maketitle

\begin{abstract}
An equiangular tight frame (ETF) is a set of unit vectors whose coherence achieves the Welch bound, and so is as incoherent as possible.
They arise in numerous applications.
It is well known that real ETFs are equivalent to a certain subclass of strongly regular graphs.
In this note, we give some alternative techniques for understanding this equivalence.
In a later document, we will use these techniques to further generalize this theory.
\end{abstract}


\section{Introduction}

Let $m\leq n$ be positive integers and let $\set{\bfphi_i}_{i=1}^{n}$ be a sequence of unit vectors in $\bbF^m$ where the field $\bbF$ is either the real line $\bbR$ or the complex plane $\bbC$.
The quantity $\max_{i\neq j}\abs{\ip{\bfphi_i}{\bfphi_j}}$ is known as the \textit{coherence} of $\set{\bfphi_i}_{i=1}^{n}$.
In many real world applications, one seeks a sequence of $n$ unit norm vectors in $\bbF^m$ whose coherence is as small as possible.
Geometrically speaking, this is equivalent to packing lines in Euclidean space: for any real unit vectors $\bfphi_i$ and $\bfphi_j$, we have $\abs{\ip{\bfphi_i}{\bfphi_j}}=\cos(\theta_{i,j})$ where $\theta_{i,j}$ is the interior angle of the lines spanned by $\bfphi_i$ and $\bfphi_j$;
finding unit vectors $\set{\bfphi_i}_{i=1}^{n}$ with minimal coherence is thus equivalent to arranging $n$ lines so that the minimum pairwise angle between any two lines is as large as possible.

Note that for a fixed $m<n$, the coherence of a sequence of $n$ unit vectors in $\bbF^m$ cannot get arbitrarily small:
the coherence is never zero (since there cannot be $n$ orthonormal vectors in $\bbF^m$) and moreover the coherence is a continuous function of $\set{\bfphi_i}_{i=1}^{n}$, which lies on the compact set of the Cartesian product of $n$ copies of the unit sphere in $\bbF^m$.
The most famous example of an explicit lower bound on the coherence is the \textit{Welch bound}~\cite{Welch74,StrohmerH03}:
\begin{theorem}
Given positive integers $m\leq n$, for any unit norm vectors $\set{\bfphi_i}_{i=1}^{n}$ in $\bbF^M$ we have
\begin{equation}
\label{equation.Welch bound}
\sqrt{\tfrac{n-m}{m(n-1)}}
\leq\max_{i\neq j}\abs{\ip{\bfphi_i}{\bfphi_j}},
\end{equation}
where equality holds if and only if $\set{\bfphi_i}_{i=1}^{n}$ is an \textit{equiangular tight frame} (ETF) for $\bbF^m$.
\end{theorem}

To fully understand this result, we first establish some notation and terminology.
For any vectors $\set{\bfphi_i}_{i=1}^{n}$ in $\bbF^m$, the corresponding \textit{synthesis operator} is the $m\times n$ matrix $\bfPhi$ which has the vectors $\set{\bfphi_i}_{i=1}^{n}$ as its columns, namely the operator $\bfPhi:\bbF^n\rightarrow\bbF^m$, $\bfPhi\bfy=\sum_{i=1}^{n}\bfy(n)\bfphi_i$.
Composing $\bfPhi$ with its $n\times m$ adjoint (conjugate transpose) $\bfPhi^*$ yields the $m\times m$ \textit{frame operator} $\bfPhi\bfPhi^*$
as well as the $n\times n$ \textit{Gram matrix} $\bfPhi^*\bfPhi$ whose $(i,j)$th entry is $(\bfPhi^*\bfPhi)(i,j)=\ip{\bfphi_i}{\bfphi_j}$.
We say $\set{\bfphi_i}_{i=1}^{n}$ is a \textit{tight frame} if $\bfPhi$ is perfectly conditioned, that is, if there exists $\alpha>0$ such that $\bfPhi\bfPhi^*=\alpha\bfI$.
This is equivalent to having the rows of $\bfPhi$ be orthogonal and equal norm.
We say $\set{\bfphi_i}_{i=1}^{n}$ is \textit{equiangular} when each $\bfphi_i$ is unit norm and the value of $\abs{\ip{\bfphi_i}{\bfphi_j}}$ is constant over all choices of $i\neq j$,
namely when the diagonal entries of $\bfPhi^*\bfPhi$ are $1$ while the off-diagonal entries have constant modulus.

An ETF is a tight frame whose vectors are equiangular.
To see why equality in~\eqref{equation.Welch bound} is achieved if and only if the unit norm vectors $\set{\bfphi_i}_{i=1}^{n}$ are an ETF,
note that if $\set{\bfphi_i}_{i=1}^{n}$ is a unit norm tight frame for $\bbF^m$ then the tight frame constant $\alpha$ is necessarily the redundancy of the frame $\tfrac{n}{m}$ since it satisfies:
\begin{equation*}
m\alpha=\Tr(\alpha\bfI)=\Tr(\bfPhi\bfPhi^*)=\Tr(\bfPhi^*\bfPhi)=\sum_{i=1}^{n}\norm{\bfphi_i}^2=\sum_{i=1}^{n}1=n.
\end{equation*}
As such, the Frobenius norm of the operator $\bfPhi\bfPhi^*-\tfrac{n}{m}\bfI$ is one way of quantifying the \textit{tightness} of $\set{\bfphi_i}_{i=1}^{n}$.
Moreover, using similar properties of the trace, we see this nonnegative quantity can be rewritten as
\begin{align*}
0
&\leq\norm{\bfPhi\bfPhi^*-\tfrac{n}{m}\bfI}_\Fro^2\\
&=\Tr(\bfPhi\bfPhi^*-\tfrac{n}{m}\bfI)^2\\
&=\Tr[(\bfPhi\bfPhi^*)^2]-2\tfrac{n}{m}\Tr(\bfPhi\bfPhi^*)+\tfrac{n^2}{m^2}\Tr(\bfI)\\
&=\Tr[(\bfPhi^*\bfPhi)^2]-2\tfrac{n}{m}\Tr(\bfPhi^*\bfPhi)+\tfrac{n^2}{m}\\
&=\sum_{i=1}^{n}\sum_{j=1}^{n}\abs{\ip{\bfphi_i}{\bfphi_j}}^2-\tfrac{n^2}{m}\\
&=\sum_{i=1}^{n}\sum_{\substack{j=1\\j\neq i}}^{n}\abs{\ip{\bfphi_i}{\bfphi_j}}^2-n(\tfrac{n}{m}-1).
\end{align*}
Bounding the above summands by their maximum then gives the inequality
\begin{equation*}
0
\leq n(n-1)\max_{i\neq j}\abs{\ip{\bfphi_i}{\bfphi_j}}^2-n(\tfrac{n}{m}-1).
\end{equation*}
When rearranged, this gives~\eqref{equation.Welch bound}.
Moreover, in order to have equality in~\eqref{equation.Welch bound}, we necessarily have equality throughout the above analysis, namely that $\norm{\bfPhi\bfPhi^*-\tfrac{n}{m}\bfI}_\Fro^2=0$ (meaning $\set{\bfphi_i}_{i=1}^{n}$ is a tight frame) and that $\abs{\ip{\bfphi_i}{\bfphi_j}}^2$ is constant over all $i\neq j$ (meaning $\set{\bfphi_i}_{i=1}^{n}$ is equiangular).
To summarize, the columns of an $m\times n$ matrix $\bfPhi$ achieve the Welch bound if and only if $\bfPhi$ has equal-norm orthogonal rows and unit-norm equiangular columns.
For a nontrivial example, consider the following $6\times 16$ real \textit{Steiner} ETF~\cite{FickusMT12} which yields an optimal packing of $16$ lines in $\bbR^6$:
\begin{equation*}
\bfPhi
=\frac1{\sqrt{3}}\left[\begin{array}{rrrrrrrrrrrrrrrr}
1&-1&\phantom{+{}}1&-1&\phantom{+{}}1&-1&\phantom{+{}}1&-1& 0& 0& 0& 0& 0& 0& 0& 0\\
0& 0& 0& 0& 0& 0& 0& 0&\phantom{+{}}1&-1&\phantom{+{}}1&-1&\phantom{+{}}1&-1&\phantom{+{}}1&-1\\
1&\phantom{+{}}1&-1&-1& 0& 0& 0& 0&\phantom{+{}}1&\phantom{+{}}1&-1&-1& 0& 0& 0& 0\\
0& 0& 0& 0&\phantom{+{}}1&\phantom{+{}}1&-1&-1& 0& 0& 0& 0&\phantom{+{}}1&\phantom{+{}}1&-1&-1\\
1&-1&-1&\phantom{+{}}1& 0& 0& 0& 0& 0& 0& 0& 0&\phantom{+{}}1&-1&-1&\phantom{+{}}1\\
0& 0& 0& 0&\phantom{+{}}1&-1&-1&\phantom{+{}}1&\phantom{+{}}1&-1&-1&\phantom{+{}}1& 0& 0& 0& 0
\end{array}\right].
\end{equation*}

Because of their minimal coherence, ETFs are useful in a number of real-world applications, including waveform design for wireless communication~\cite{StrohmerH03}, compressed sensing~\cite{DeVore07} and algebraic coding theory~\cite{JasperMF14}.
In spite of this fact, only a few methods for constructing ETFs are known.
Real ETFs in particular are equivalent to a certain class of very symmetric graphs known as \textit{strongly regular graphs} (SRGs).
Much of the work behind this equivalence was pioneered by J.~J.~Seidel and his contemporaries~\cite{CorneilM91},
and a nice, concise discussion of this mathematics was recently given by Waldon~\cite{Waldron09}.
For the frame community, this equivalence is invaluable since it allows us to leverage the rich SRG literature,
notably the SRG existence tables in a book chapter~\cite{Brouwer07} and website~\cite{Brouwer14} by Brouwer.

In this short paper, we provide some alternative techniques for understanding this equivalence.
Rather than emphasize the spectra of certain matrices, we instead just focus on quadratic relationships that they satisfy.
Three minor, but apparently novel contributions of this work are (i) a closed form expression for the dimensions of an ETF in terms of the parameters of an SRG, (ii) a proof that a given SRG can only lead to one ETF under the standard means of identifying them, and (iii) a realization that real ETFs correspond to SRGs whose parameters satisfy $\mu=\frac k2$ (i.e.\ that restrictions on the parameter $\lambda$ are superfluous).
In the next section, we review the basic properties of SRGs in general.
In the third and final section, we discuss the equivalence between real ETFs and certain SRGs.
The remainder of this section discusses some other basic facts about ETFs that we will use later on.

Most known constructions of ETFs fall into one of two categories: constructions of synthesis operators~\cite{XiaZG05,DingF07,FickusMT12,JasperMF14}, and constructions of Gram matrices~\cite{StrohmerH03,HolmesP04,Waldron09}.
In this paper, we focus on the latter approach.
It is especially attractive in the real setting, since the off-diagonal entries of the Gram matrix of a real ETF can only be two possible values, namely the Welch bound and its negative.
Here, the key idea is the following result:

\begin{lemma}
\label{lemma.Gram matrix ETF characterization}
An $n\times n$ self-adjoint matrix $\bfG$ is the Gram matrix of an ETF if and only if (i) $\bfG^2=\alpha\bfG$ for some $\alpha\in\bbF$, (ii) $\bfG(i,i)=1$ for all $i$, and (iii) there exists $\beta\in\bbF$ such that $\abs{\bfG(i,j)}=\beta$ for all $i\neq j$.
\end{lemma}

\begin{proof}
($\Rightarrow$) If $\bfG=\bfPhi^*\bfPhi$ where $\bfPhi$ is the synthesis operator of an ETF $\set{\bfphi_i}_{i=1}^{n}$, then $\bfG$ immediately satisfies (ii) and (iii), while (i) follows from the fact that $\set{\bfphi_i}_{i=1}^{n}$ is tight: $\bfG^2=(\bfPhi^*\bfPhi)^2=\bfPhi^*(\alpha\bfI)\bfPhi=\alpha\bfG$.
($\Leftarrow$) Assume $\bfG$ satisfies (i), (ii) and (iii).
Since $\bfG^2=\alpha\bfG$, the only possible eigenvalues of $\bfG$ are $\alpha$ and $0$.
Letting $m$ denote the multiplicity of $\alpha$, the fact that $\bfG$ is self-adjoint (and thus normal) implies there exists an $n\times n$ unitary matrix $\bfU$ and a diagonal matrix $\bfD$ such that
\begin{equation*}
\bfG
=\bfU\bfD\bfU^*
=\left[\begin{array}{ll}\bfU_1&\bfU_2\end{array}\right]\left[\begin{array}{ll}\alpha\bfI&\bfzero\\\bfzero&\bfzero\end{array}\right]\left[\begin{array}{l}\bfU_1^*\\\bfU_2^*\end{array}\right]
=\alpha\bfU_1^{}\bfU_1^*,
\end{equation*}
where $\bfU_1$ and $\bfU_2$ are $n\times m$ and $n\times(n-m)$ submatrices of $\bfU$ whose columns are eigenvectors of $\bfG$ with eigenvalues $\alpha$ and $0$, respectively.
Note that since $m\alpha=\Tr(\bfG)=\sum_{i=1}^{n}\bfG(i,i)=n$ we know that $m$ not zero, and so is positive (being a multiplicity).
Thus, $\alpha=\frac nm$ is nonnegative.
This allows us to let $\set{\bfphi_i}_{i=1}^{n}$ be the columns of the $m\times n$ matrix $\bfPhi=\sqrt{\alpha}\,\bfU_1^*$.
Since the columns of $\bfU_1$ are orthonormal we have $\bfPhi\bfPhi^*=\alpha\bfU_1^*\bfU_1^{}=\alpha\bfI$ and so these vectors form a tight frame for $\bbR^m$.
Moreover, these vectors are an ETF since their Gram matrix $\bfPhi^*\bfPhi=\alpha\bfU_1^{}\bfU_1^*=\bfG$ satisfies (ii) and (iii).
\end{proof}

We will also need the well known fact that every $m\times n$ ETF has complementary $(n-m)\times n$ ETFs.
These ETFs are often called \textit{Naimark complements} since the proof of their existence is similar to Naimark's dilation theorem which, from a frame theory perspective, states that any tight frame is a scaled orthogonal projection of an orthonormal basis.
In the finite-dimensional setting of this paper, all such complements can be constructed using elementary linear algebra.

\begin{lemma}
Let $\bfPhi$ be an $m\times n$ synthesis operator of an ETF, and let $\tilde{\bfPhi}$ be any $(n-m)\times n$ matrix whose rows form an orthogonal basis for the orthogonal complement of the row space of $\bfPhi$, and have squared norm $n/(n-m)$.
Then $\tilde{\bfPhi}$ is the synthesis operator of an ETF and moreover, the two Gram matrices satisfy
\begin{equation*}
\bfI
=\tfrac{m}{n}\bfPhi^*\bfPhi+\tfrac{n-m}{n}\tilde{\bfPhi}^*\tilde{\bfPhi}.
\end{equation*}
\end{lemma}

\begin{proof}
Since $\bfPhi$ is the synthesis matrix of an ETF we have \smash{$\bfPhi\bfPhi^*=\frac nm\bfI$}.
Thus, the rows of \smash{$(\frac{m}{n})^{\frac12}\,\bfPhi$} are orthonormal.
For any matrix $\tilde{\bfPhi}$ that satisfies our hypotheses, we have \smash{$\tilde{\bfPhi}\tilde{\bfPhi}^*=\frac{n}{m-m}\bfI$},
meaning its columns $\set{\tilde{\bfphi}_i}_{i=1}^{n}$ form a tight frame for $\bbF^{n-m}$.
Moreover, the matrix
\begin{equation*}
\left[\begin{array}{r}
(\frac{m}{n})^{\frac12}\,\bfPhi\\
(\frac{n-m}{n})^{\frac12}\,\tilde{\bfPhi}
\end{array}\right]
\end{equation*}
is square and has orthonormal rows.
It is thus unitary, and so also has orthonormal columns:
\begin{equation*}
\bfI
=
\left[\begin{array}{rr}
(\frac{m}{n})^{\frac12}\,\bfPhi^*&
(\frac{n-m}{n})^{\frac12}\,\tilde{\bfPhi}^*
\end{array}\right]
\left[\begin{array}{r}
(\frac{m}{n})^{\frac12}\,\bfPhi\\
(\frac{n-m}{n})^{\frac12}\,\tilde{\bfPhi}
\end{array}\right]
=\tfrac{m}{n}\bfPhi^*\bfPhi+\tfrac{n-m}{n}\tilde{\bfPhi}^*\tilde{\bfPhi}.
\end{equation*}
In particular, the Gram matrix of $\set{\tilde{\bfphi}_i}_{i=1}^{n}$ is \smash{$\tilde{\bfPhi}^*\tilde{\bfPhi}=\tfrac1{n-m}(n\bfI-m\bfPhi^*\bfPhi)$}.
Since the diagonal entries of $\bfPhi^*\bfPhi$ are $1$ while its off-diagonal entries have constant modulus, the matrix \smash{$\tilde{\bfPhi}^*\tilde{\bfPhi}$} has these same properties, and so $\set{\tilde{\bfphi}_i}_{i=1}^{n}$ is an ETF for $\bbF^{n-m}$.
\end{proof}

As we shall see in the coming sections, in the real case, the structure of an ETF is entirely encoded by the pattern of positive and negative values that lie on the off-diagonal of its Gram matrix.
Taking a Naimark complement simply changes the signs of these values.

\section{Basic facts about strongly regular graphs}

As discussed in the introduction, the Gram matrix $\bfG$ of a real ETF is a real symmetric matrix whose diagonal entries are $1$ while its off-diagonal entries are either the Welch bound or its negative.
As detailed in the next section, $\bfG$ can easily be converted into the adjacency matrix $\bfA$ of a graph.
Moreover, since $\bfG^2=\alpha\bfG$, it is reasonable to believe that $\bfA$ also satisfies some quadratic relationship, which in turn implies the graph possesses certain symmetries.

In general, let $\bfA$ be the \textit{adjacency matrix} of a graph on $v$ vertices, namely a $v\times v$ real symmetric matrix whose entries have value either $0$ or $1$, and whose diagonal entries are all $0$.
The corresponding graph is \textit{regular} if all vertices in the graph have the same number of neighbors, namely if there exists a nonnegative integer $k$ such that $\bfA\bfone=k\bfone$ where $\bfone$ is a $v\times 1$ vectors of ones.
Such a graph is said to be \textit{strongly regular} with nonnegative integer parameters $(v,k,\lambda,\mu)$ if any two neighbors have exactly $\lambda$ neighbors in common while any two nonneighbors have exactly $\mu$ neighbors in common.
A strongly regular graph (SRG) with such parameters is often denoted a $\SRG(v,k,\lambda,\mu)$.

Fortunately, this nonintuitive definition of an SRG has a simpler algebraic characterization:
since $\bfA^2(i,j)$ counts the number of two-step paths from vertex $i$ to $j$, a given graph is a $\SRG(v,k,\lambda,\mu)$ if and only if
\begin{equation}
\label{equation.pre definition of srg}
\bfA^2(i,j)=\left\{\begin{array}{cl}k,&i=j,\\\lambda,&i\neq j, \bfA(i,j)=1,\\\mu,&i\neq j, \bfA(i,j)=0,\end{array}\right.
\end{equation}
namely if and only if
\begin{equation}
\label{equation.definition of srg}
\bfA^2
=(\lambda-\mu)\bfA+(k-\mu)\bfI+\mu\bfJ
\end{equation}
where $\bfJ=\bfone\bfone^*$ is a $v\times v$ matrix of ones.
Because of its simplicity, we take~\eqref{equation.definition of srg} as our definition of an SRG.
To show that a given adjacency matrix $\bfA$ corresponds to an SRG, note it suffices to show that there exists real numbers $x,y,z$ such that $\bfA^2=x\bfA+y\bfI+z\bfJ$.
Indeed, in this case we can define $k=y+z$, $\lambda=x+z$ and $\mu=z$ to obtain~\eqref{equation.definition of srg} which is equivalent to~\eqref{equation.pre definition of srg}; excluding $\bfA=\bfzero$ and $\bfA=\bfJ-\bfI$, both of which are trivially SRGs, we have that each number $k$, $\lambda$ and $\mu$ appears at least once in $\bfA^2$, proving they are nonnegative integers.

Note any adjacency matrix that satisfies \eqref{equation.definition of srg} is necessarily regular, since $\bfA^2(i,i)$ counts the number of two-step paths from vertex $i$ to itself:
\begin{equation*}
k
=\bfA^2(i,i)
=\sum_{j=1}^{v}\bfA(i,j)\bfA(j,i)
=\sum_{j=1}^{v}[\bfA(i,j)]^2
=\sum_{j=1}^{v}\bfA(i,j)
=(\bfA\bfone)(i).
\end{equation*}
That is, $\bfA\bfone=k\bfone$.
Since we also know $\bfA^*=\bfA$, we can conjugate~\eqref{equation.definition of srg} by the vector $\bfone$ to obtain:
\begin{align*}
k^2v
&=(k\bfone)^*k\bfone\\
&=(\bfA\bfone)^*(\bfA\bfone)\\
&=\bfone^*\bfA^2\bfone\\
&=(\lambda-\mu)\bfone^*\bfA\bfone+(k-\mu)\bfone^*\bfI\bfone+\mu\bfone^*\bfJ\bfone\\
&=(\lambda-\mu)\bfone^*(k\bfone)+(k-\mu)\bfone^*\bfone+\mu(\bfone^*\bfone)^2\\
&=(\lambda-\mu)kv+(k-\mu)v+\mu v^2.
\end{align*}
Dividing by $v$ and collecting common terms gives an implicit condition our four SRG parameters:
\begin{equation}
\label{equation.srg parameter relation}
k(k-\lambda-1)
=(v-k-1)\mu.
\end{equation}
The SRG condition~\eqref{equation.definition of srg} also completely determines the spectrum of our adjacency matrix $\bfA$:
$k$ is an eigenvalue of $\bfA$ with eigenvector $\bfone$, and if $\bfA\bfx=\gamma\bfx$ where $\bfx$ is nonzero and orthogonal to $\bfone$, then applying~\eqref{equation.definition of srg} to $\bfx$ gives
\begin{equation*}
\gamma^2\bfx
=\bfA^2\bfx
=(\lambda-\mu)\bfA\bfx+(k-\mu)\bfI\bfx+\mu\bfone\bfone^*\bfx
=[(\lambda-\mu)\gamma+(k-\mu)]\bfx.
\end{equation*}
Thus, the remaining $v-1$ eigenvalues of $\bfA$ are roots of the quadratic equation $\gamma^2-(\lambda-\mu)\gamma-(k-\mu)$, namely
\begin{equation*}
\gamma_+=\tfrac12\bigbracket{(\lambda-\mu)+\sqrt{(\lambda-\mu)^2+4(k-\mu)}},
\quad
\gamma_-=\tfrac12\bigbracket{(\lambda-\mu)-\sqrt{(\lambda-\mu)^2+4(k-\mu)}}.
\end{equation*}
To compute the multiplicities $m_+$ and $m_-$ of $\gamma_+$ and $\gamma_-$,
we write $m_+=\tfrac12(v-1)-x$ and $m_-=\tfrac12(v-1)+x$ for some $x\in\bbR$, and take the trace of $\bfA$ to find
\begin{equation*}
0
=\Tr(A)
=k+[\tfrac12(v-1)-x]\gamma_++[\tfrac12(v-1)+x]\gamma_-
=k+\tfrac12(v-1)(\gamma_++\gamma_-)-x(\gamma_+-\gamma_-).
\end{equation*}
Solving for $x$ gives
\begin{equation*}
m_+
=\frac12\biggbracket{(v-1)-\frac{2k+(v-1)(\lambda-\mu)}{\sqrt{(\lambda-\mu)^2+4(k-\mu)}}},
\qquad
m_-
=\frac12\biggbracket{(v-1)+\frac{2k+(v-1)(\lambda-\mu)}{\sqrt{(\lambda-\mu)^2+4(k-\mu)}}}.
\end{equation*}

The complement of an SRG is another SRG.
To be precise, the \textit{graph complement} of a given graph is obtained by disconnecting neighbors and connecting nonneighbors,
namely by considering the adjacency matrix $\tilde{\bfA}=\bfJ-\bfA-\bfI$.
If $\bfA\bfone=k\bfone$ then $\tilde{\bfA}\bfone=(\bfJ-\bfA-\bfI)\bfone=(v-k-1)\bfone$, meaning $\tilde{\bfA}$ is regular with degree $v-k-1$.
In this case, $\tilde{\bfA}\bfJ=\tilde{\bfA}\bfone\bfone^*=(v-k-1)\bfone\bfone^*=(v-k-1)\bfJ$ and $\bfJ\tilde{\bfA}=\bfJ^*\tilde{\bfA}^*=(\tilde{\bfA}\bfJ)^*=[(v-k-1)\bfJ]^*=(v-k-1)\bfJ$.
Having these facts, note that if $\bfA$ is a $\SRG(v,k,\lambda,\mu)$ then substituting $\bfA=\bfJ-\tilde{\bfA}-\bfI$ into~\eqref{equation.definition of srg} gives
\begin{align*}
(\mu-\lambda)\tilde{\bfA}+(k-\lambda)\bfI+\lambda\bfJ
&=(\lambda-\mu)(\bfJ-\tilde{\bfA}-\bfI)+(k-\mu)\bfI+\mu\bfJ\\
&=(\lambda-\mu)\bfA+(k-\mu)\bfI+\mu\bfJ\\
&=\bfA^2\\
&=(\bfJ-\tilde{\bfA}-\bfI)^2\\
&=\bfJ^2+\tilde{\bfA}^2+\bfI-\bfJ\bfA-\bfA\bfJ-2\bfJ+2\tilde{\bfA}\\
&=\tilde{\bfA}^2+2\tilde{\bfA}+\bfI+[v-2(v-k-1)-2]\bfJ\\
&=\tilde{\bfA}^2+2\tilde{\bfA}+\bfI-(v-2k)\bfJ.
\end{align*}
Solving for $\tilde{\bfA}^2$ then gives
\begin{equation*}
\tilde{\bfA}^2
=(\mu-\lambda-2)\tilde{\bfA}+(k-\lambda-1)\bfI+(v-2k+\lambda)\bfJ.
\end{equation*}
This means that $\tilde{\bfA}$ satisfies an equation of the form~\eqref{equation.definition of srg}, namely
$\tilde{\bfA}^2=(\tilde{\lambda}-\tilde{\mu})\tilde{\bfA}+(\tilde{k}-\tilde{\mu})\bfI+\tilde{\mu}\bfJ$ where
\begin{align*}
\tilde{\mu}&=v-2k+\lambda,\\
\tilde{\lambda}&=(\tilde{\lambda}-\tilde{\mu})+\tilde{\mu}=(\mu-\lambda-2)+(v-2k+\lambda)=v-2k+\mu-2,\\
\tilde{k}&=(\tilde{k}-\tilde{\mu})+\tilde{\mu}=(k-\lambda-1)+(v-2k+\lambda)=v-k-1.
\end{align*}
To summarize, the graph complement of a $\SRG(v,k,\lambda,\mu)$ is a $\SRG(v,v-k-1,v-2k+\mu-2,v-2k+\lambda)$.
In Brouwer's SRG tables~\cite{Brouwer07,Brouwer14}, a graph and its complement always appear in subsequent rows.

\section{Equating real ETFs and certain SRGs}

As seen in the previous section, an SRG is just an adjacency matrix that satisfies a quadratic relation~\eqref{equation.definition of srg}.
In this section, we discuss the traditional method for obtaining an SRG from a real ETF, and vice versa.

To begin, assume $\set{\bfphi_i}_{i=1}^{n}$ is an ETF for $\bbR^m$ with $m<n$, and let $\alpha=\tfrac{n}{m}$ and \smash{$\beta=[\frac{n-m}{m(n-1)}]^{\frac12}$} denote its redundancy and Welch bound, respectively.
From Lemma~\ref{lemma.Gram matrix ETF characterization}, we know its $n\times n$ Gram matrix $\bfG=\bfPhi^*\bfPhi$ satisfies $\bfG^2=\alpha\bfG$, has ones along its diagonal, and has values of $\pm\beta$ off its diagonal.
We can thus convert $\bfG$ into an adjacency matrix $\bfA$ by changing its diagonal entries to zero while changing the off-diagonal values of $\beta$ and $-\beta$ to $1$ and $0$, respectively.
That is, we let
\begin{equation}
\label{equation.ETF to SRG conversion}
\bfA=\tfrac1{2\beta}\bfG-\tfrac{\beta+1}{2\beta}\bfI+\tfrac12\bfJ.
\end{equation}
(A note: this is a slight departure from the existing literature in which $\beta$ and $-\beta$ are instead changed to $0$ and $1$ respectively~\cite{Waldron09}.
We make this change because if we view the frame vectors as points on a sphere, it is geometrically more natural to identify points as neighbors when the angle between them is acute, as opposed to obtuse.
Regardless, this decision is of little consequence mathematically, since making the other identification simply results in the complement of~\eqref{equation.ETF to SRG conversion}, which corresponds to an SRG if and only if~\eqref{equation.ETF to SRG conversion} does.)

To see whether $\bfA$ corresponds to an SRG, recall we must only determine whether it satisfies a relation of the form $\bfA^2=x\bfA+y\bfI+z\bfJ$ for some real scalars $x,y,z$.
This is plausible since $\bfG^2=\alpha\bfG$.
There is a problem, however: the number of $\beta$'s per row of $\bfG$ is not constant, meaning that when we attempt to compute $\bfA^2$, we have no way of simplifying the $\bfG\bfJ$ term.
In short, for an ETF in general, the graph defined by \eqref{equation.ETF to SRG conversion} need not be regular, let alone strongly regular.
Here, the standard remedy is to change the signs of the frame vectors so that the resulting graph has a strongly regular subgraph of $v=n-1$ vertices.

To elaborate, one can negate any number of the vectors $\set{\bfphi_i}_{i=1}^{n}$ to obtain a \textit{switching equivalent} ETF.
This corresponds to multiplying the $m\times n$ synthesis operator $\bfPhi$ on the right by an $n\times n$ diagonal matrix $\bfD$ whose diagonal entries are $\pm1$.
Note the resulting frame is still tight since $\bfPhi\bfD\bfD^*\bfPhi^*=\bfPhi\bfPhi^*=\alpha\bfI$.
Moreover, it is still equiangular since is Gram matrix $\bfD^*\bfPhi^*\bfPhi\bfD$ is obtained from $\bfPhi^*\bfPhi$ by multiplying some paired rows and columns by $-1$.
As such, given an $m\times n$ real ETF, we may negate $\bfphi_i$'s as necessary so as to assume, without loss of generality, that $\ip{\bfphi_1}{\bfphi_i}=\beta$ for all $i=2,\dotsc,n$.
Here, the resulting matrix~\eqref{equation.ETF to SRG conversion} is of the form
\begin{equation}
\label{equation.traditional SRG 1}
\bfA=\left[\begin{array}{ll}0&\bfone^*\\\bfone&\bfB\end{array}\right]
\end{equation}
where $\bfB$ is a $v\times v$ adjacency matrix.
Solving for the $\bfG$ in~\eqref{equation.ETF to SRG conversion} then gives
\begin{align}
\nonumber
\bfG
&=2\beta\bfA+(\beta+1)\bfI-\beta\bfJ\\
\nonumber
&=2\beta\left[\begin{array}{ll}0&\bfone^*\\\bfone&\bfB\end{array}\right]+(\beta+1)\left[\begin{array}{ll}1&\bfzero^*\\\bfzero&\bfI\end{array}\right]-\beta\left[\begin{array}{ll}1&\bfone^*\\\bfone&\bfJ\end{array}\right]\\
\label{equation.traditional SRG 2}
&=\left[\begin{array}{ll}1&\beta\bfone^*\\\beta\bfone&2\beta\bfB+(\beta+1)\bfI-\beta\bfJ\end{array}\right].
\end{align}
Substituting this expression for $\bfG$ into the relation $\bfG^2=\alpha\bfG$ then gives
\begin{equation}
\label{equation.traditional SRG 3}
\left[\begin{array}{ll}1&\beta\bfone^*\\\beta\bfone&2\beta\bfB+(\beta+1)\bfI-\beta\bfJ\end{array}\right]\left[\begin{array}{ll}1&\beta\bfone^*\\\beta\bfone&2\beta\bfB+(\beta+1)\bfI-\beta\bfJ\end{array}\right]
=\left[\begin{array}{ll}\alpha&\alpha\beta\bfone^*\\\alpha\beta\bfone&2\alpha\beta\bfB+\alpha(\beta+1)\bfI-\alpha\beta\bfJ\end{array}\right].
\end{equation}
Multiplying out the left-hand side of~\eqref{equation.traditional SRG 3} and then equating the upper-left terms gives $\alpha=v\beta^2+1$, something that also quickly follows from the fact that $v=n-1$, $\beta=[\frac{n-m}{m(n-1)}]^{\frac12}$ and $\alpha=\tfrac{n}{m}$.
Meanwhile, equating the lower-left terms of~\eqref{equation.traditional SRG 3} gives $\alpha\beta\bfone=\beta\bfone+2\beta^2\bfB\bfone+\beta(\beta+1)\bfone-v\beta^2\bfone$, namely
\begin{equation*}
\bfB\bfone=(\tfrac{v-1}{2}+\tfrac{\alpha-2}{2\beta})\bfone.
\end{equation*}
This means $\bfB$ corresponds to a regular graph of degree $k=\tfrac{v-1}{2}+\tfrac{\alpha-2}{2\beta}=\tfrac n2-1+(\tfrac n{2m}-1)[\tfrac{m(n-1)}{n-m}]^{\frac12}$.
We mention that the integrality of this $k$ is by no means obvious, and is in fact closely related to strong, previously known integrality conditions on the existence of real ETFs~\cite{HolmesP04,SustikTDH07,Waldron09}.
Since $\bfB\bfone=k\bfone$, $\bfB\bfJ=\bfB\bfone\bfone^*=k\bfone\bfone^*=k\bfJ$, at which point the symmetry of $\bfB$ and $\bfJ$ also gives $\bfJ\bfB=k\bfJ$.
This allows us to simplify the lower-right term of~\eqref{equation.traditional SRG 3}:
\begin{align}
\nonumber
2\alpha\beta\bfB+\alpha(\beta+1)\bfI-\alpha\beta\bfJ
&=\beta^2\bfJ+(2\beta\bfB+(\beta+1)\bfI-\beta\bfJ)^2\\
\nonumber
&=\beta^2\bfJ+4\beta^2\bfB^2+(\beta+1)^2\bfI+v\beta^2\bfJ+4\beta(\beta+1)\bfB-2\beta(\beta+1)\bfJ-2\beta^2\bfB\bfJ-2\beta^2\bfJ\bfB\\
\label{equation.traditional SRG 4}
&=4\beta^2\bfB^2+4\beta(\beta+1)\bfB+(\beta+1)^2\bfI+[(v-4k-1)\beta^2-2\beta]\bfJ.
\end{align}
Solving for $\bfB^2$ then gives an expression of the form $\bfB^2=x\bfB+y\bfI+z\bfJ$, namely that $\bfB$ is strongly regular:
\begin{equation*}
\bfB^2
=\tfrac{\alpha-2\beta-2}{2\beta}\bfB+\tfrac{(\alpha+\beta+1)(\beta+1)}{4\beta^2}\bfI-\tfrac{(\alpha-2)+(v-4k-1)\beta}{4\beta}\bfJ.
\end{equation*}
To explicitly determine the parameters of this SRG, recall that $\bfB^2=(\lambda-\mu)\bfB+(k-\mu)\bfI+\mu\bfJ$ and rewrite $k=\tfrac{v-1}{2}+\tfrac{\alpha-2}{2\beta}$ as $\alpha-2=(2k-v+1)\beta$ to obtain
\begin{equation*}
\mu
=-\tfrac{(\alpha-2)+\beta(v-4k-1)}{4\beta}
=-\tfrac{(2k-v+1)\beta+(v-4k-1)\beta}{4\beta}
=\tfrac{k}{2}.
\end{equation*}
This, in turn, immediately gives determines $\lambda$ according to~\eqref{equation.srg parameter relation}:
\begin{equation}
\label{equation.traditional SRG 5}
k(k-\lambda-1)
=(v-k-1)\mu
=(v-k-1)\tfrac{k}{2}
\quad
\Longrightarrow
\quad
\lambda
=(k-1)-\tfrac{v-k-1}{2}
=\tfrac{3k-v-1}{2}.
\end{equation}
We summarize these previously known facts in the following theorem:

\begin{theorem}
\label{theorem.n-1 ETF to SRG}
Let $\set{\bfphi_i}_{i=1}^{n}$ be an ETF for $\bbR^m$ where $m<n$, and assume without loss of generality that $\ip{\bfphi_1}{\bfphi_i}>0$ for all~$i$.
Letting $\alpha=\tfrac{n}{m}$ and $\beta=[\frac{n-m}{m(n-1)}]^{\frac12}$, the $(n-1)\times(n-1)$ matrix $\bfB$ such that
\begin{equation*}
\left[\begin{array}{ll}0&\bfone^*\\\bfone&\bfB\end{array}\right]
=\bfA
=\tfrac1{2\beta}\bfPhi^*\bfPhi-\tfrac{\beta+1}{2\beta}\bfI+\tfrac12\bfJ,
\end{equation*}
is the adjacency matrix of a $\SRG(v,k,\lambda,\mu)$ where $v=n-1$, $k=\tfrac n2-1+(\tfrac n{2m}-1)[\tfrac{m(n-1)}{n-m}]^{\frac12}$ and $\mu=\tfrac k2$.
\end{theorem}

In summary, every $n$-vector real ETF yields an SRG on $v=n-1$ vertices where $\mu=\frac k2$.
For an example of this result, consider $m=7$ and $n=28$.
A real $7\times 28$ real ETF indeed exists.
It can be constructed, for example, as a Steiner ETF~\cite{FickusMT12} arising from the \textit{Fano plane}, a special type of finite geometry.
Here, $\alpha=4$ and $\beta=\frac13$, and so Theorem~\ref{theorem.n-1 ETF to SRG} tells us that there exists an SRG with $v=27$, $k=16$ and $\mu=8$
(at which point~\eqref{equation.srg parameter relation} gives $\lambda=10$).
Indeed, consulting a table of SRGs~\cite{Brouwer07,Brouwer14}, a $\SRG(27,16,10,8)$ is known to exist.
This begs the question: if we only knew that a $\SRG(27,16,10,8)$ exists, could we reverse the above argument to prove that a $7\times28$ real ETF exists?
As we now detail, the answer is yes.
In fact, we can convert any SRG into a real ETF provided $\mu=\frac k2$.

To be precise, now let $\bfB$ be the $v\times v$ adjacency matrix of any given $\SRG(v,k,\lambda,\mu)$ for which $\mu=\frac k2$.
Let $n=v+1$ and let~\eqref{equation.traditional SRG 1} define an $n\times n$ adjacency matrix $\bfA$.
For any $\beta\in\bbR$, we can use~\eqref{equation.traditional SRG 2} to define an $n\times n$ real symmetric matrix $\bfG$ whose diagonal entries are all $1$ and whose off-diagonal entries are either $\pm\beta$.
We claim that for an appropriate choice of $\beta$, this matrix $\bfG$ is the Gram matrix of a $m\times n$ ETF for a certain choice of $m$.
In light of Lemma~\ref{lemma.Gram matrix ETF characterization}, this reduces to finding scalars $\alpha,\beta$ so that the matrix $\bfG$ defined in~\eqref{equation.traditional SRG 2} satisfies $\bfG^2=\alpha\bfG$, namely~\eqref{equation.traditional SRG 3}.

Note that if $\beta=0$ then $\bfG=\bfI$ is the Gram matrix of an orthonormal sequence of vectors, which is a trivial type of ETF.
As such, we assume $\beta\neq0$ from this point forward.
Parallelling our earlier discussion, satisfying the upper-left part of~\eqref{equation.traditional SRG 3} is equivalent to having
\begin{equation}
\label{equation.traditional SRG 7}
\alpha=v\beta^2+1.
\end{equation}
Next, since $\bfB$ is a $\SRG(v,k,\lambda,\mu)$ we have $\bfB\bfone=k\bfone$.
As such, satisfying the lower-left and upper-right parts of~\eqref{equation.traditional SRG 3} is equivalent to having
\begin{equation*}
\alpha\beta\bfone
=\beta\bfone+2\beta^2\bfB\bfone+\beta(\beta+1)\bfone-v\beta^2\bfone
=[-(v-2k-1)\beta^2+2\beta]\bfone,
\end{equation*}
that is, to
\begin{equation}
\label{equation.traditional SRG 8}
\alpha=-(v-2k-1)\beta+2.
\end{equation}
We next claim that~\eqref{equation.traditional SRG 7} and~\eqref{equation.traditional SRG 8} automatically imply the lower-right part of~\eqref{equation.traditional SRG 3}.
To see this, note that in light of~\eqref{equation.traditional SRG 4}, we want to show that
\begin{equation}
\label{equation.traditional SRG 6}
2\alpha\beta\bfB+\alpha(\beta+1)\bfI-\alpha\beta\bfJ
=4\beta^2\bfB^2+4\beta(\beta+1)\bfB+(\beta+1)^2\bfI+[(v-4k-1)\beta^2-2\beta]\bfJ.
\end{equation}
To simplify this further, note that since $\bfB$ is a $\SRG(v,k,\lambda,\mu)$ we have $\bfB^2=(\lambda-\mu)\bfB+(k-\mu)\bfI+\mu\bfJ$.
Moreover, since $\mu=\frac k2$, \eqref{equation.traditional SRG 5} gives $\lambda=\tfrac{3k-v-1}{2}$ and so
$\bfB^2=-\tfrac{v-2k+1}2\bfB+\tfrac k2\bfI+\tfrac k2\bfJ$, making~\eqref{equation.traditional SRG 6} equivalent to
\begin{align*}
2\alpha\beta\bfB+\alpha(\beta+1)\bfI-\alpha\beta\bfJ
&=4\beta^2[-\tfrac{v-2k+1}2\bfB+\tfrac k2\bfI+\tfrac k2\bfJ]+4\beta(\beta+1)\bfB+(\beta+1)^2\bfI+[(v-4k-1)\beta^2-2\beta]\bfJ\\
&=2\beta[-(v-2k-1)\beta+2]\bfB+[2k\beta^2+(\beta+1)^2]\bfI+\beta[-(v-2k-1)\beta+2]\bfJ.
\end{align*}
Equating the coefficients of $\bfB$, $\bfI$ and $\bfJ$, it thus suffices to meet the three conditions:
\begin{align*}
2\alpha\beta&=2\beta[-(v-2k-1)\beta+2],\\
\alpha(\beta+1)&=2k\beta^2+(\beta+1)^2,\\
\alpha\beta&=\beta[-(v-2k-1)\beta+2].
\end{align*}
The first and last of these conditions are immediately obtained by multiplying~\eqref{equation.traditional SRG 7} by $2\beta$ and $\beta$, respectively.
Meanwhile, the middle condition is implied by a combination of~\eqref{equation.traditional SRG 7} and~\eqref{equation.traditional SRG 8}:
\begin{equation*}
\alpha(\beta+1)
=\alpha\beta+\alpha
=[-(v-2k-1)\beta^2+2\beta]+(v\beta^2+1)
=(2k+1)\beta^2+2\beta+1
=2k\beta^2+(\beta+1)^2
\end{equation*}

To summarize, given the incidence matrix $\bfB$ of a $(v,k,\lambda,\mu)$-SRG, any scalars $\alpha,\beta$ that satisfy~\eqref{equation.traditional SRG 7} and~\eqref{equation.traditional SRG 8} will yield the Gram matrix $\bfG$ of an ETF via~\eqref{equation.traditional SRG 2}.
These conditions on $\alpha$ and $\beta$ are equivalent to defining $\alpha=v\beta^2+1$ where $\beta$ is any root of the quadratic
\begin{equation*}
v\beta^2+(v-2k-1)\beta-1=0.
\end{equation*}
That is,
\begin{equation}
\label{equation.traditional SRG 9}
\beta
=\tfrac1{2v}\bigset{-(v-2k-1)\pm[(v-2k-1)^2+4v]^{\frac12}}
=\tfrac1{2v}[-\delta\pm(\delta^2+4v)^{\frac12}],
\end{equation}
where we introduce the notion of the \textit{deviation} $\delta=v-2k-1$ of an SRG.
This means that for any SRG with nonzero deviation, there are two choices of $\beta$ that make~\eqref{equation.traditional SRG 2} into an ETF, one positive and the other negative.
This makes sense, since every ETF has a Naimark complement and their two Gram matrices have opposite sign patterns.
For more confirmation, we use~\eqref{equation.traditional SRG 9} to compute the dimension $m$ of space in which our ETF lies:
since $m$ is the multiplicity of $\alpha$ as an eigenvalue of $\bfG$, \eqref{equation.traditional SRG 8} and~\eqref{equation.traditional SRG 9} give
\begin{equation}
\label{equation.traditional SRG 10}
v+1
=n
=\Tr(\bfG)
=m\alpha
=m(v\beta^2+1),
\end{equation}
and so
$m=\frac{v+1}{v\beta^2+1}=\frac{2v(v+1)}{(\delta^2+4v)\mp\delta(\delta^2+4v)^{\frac12}}$.
To simplify this further, multiply by the conjugate of the denominator:
\begin{equation}
\label{equation.traditional SRG 11}
m
=\tfrac{2v(v+1)[(\delta^2+4v)\pm\delta(\delta^2+4v)^{\frac12}]}{(\delta^2+4v)^2-\delta^2(\delta^2+4v)}
=\tfrac{2v(v+1)[(\delta^2+4v)\pm\delta(\delta^2+4v)^{\frac12}]}{4v(\delta^2+4v)^2}
=\tfrac{v+1}{2}\bigparen{1\pm\tfrac{\delta}{(\delta^2+4v)^{\frac12}}}.
\end{equation}
Note these two possibilities for $m$ add up to $n=v+1$, just as we expect the ambient dimension of ETF and its Naimark complement to do.
To ensure that $\beta$ is the Welch bound, we choose ``$+$" in both~\eqref{equation.traditional SRG 9} and~\eqref{equation.traditional SRG 11}.
Indeed, since $m$ satisfies~\eqref{equation.traditional SRG 10} and $\alpha=v\beta^2+1$,
we know $\alpha=\frac nm$ and $\beta^2=\frac{\alpha-1}{n-1}=\frac{n-m}{m(n-1)}$; choosing ``$+$" gives $\beta>0$ and so \smash{$\beta=[\frac{n-m}{m(n-1)}]^{\frac12}$}.
We summarize these facts as follows:

\begin{theorem}
\label{theorem.n-1 SRG to ETF}
Let $\bfB$ the adjacency matrix of an $\SRG(v,k,\lambda,\mu)$ with $\mu=\frac k2$.
Then
\begin{equation*}
m=\tfrac{v+1}{2}\bigset{1+\tfrac{v-2k-1}{[(v-2k-1)^2+4v]^{\frac12}}}
\end{equation*}
is the unique choice of $m$ for which there exists $\beta>0$ such that
\begin{equation*}
\bfG
=\left[\begin{array}{ll}1&\beta\bfone^*\\\beta\bfone&2\beta\bfB+(\beta+1)\bfI-\beta\bfJ\end{array}\right],
\end{equation*}
is the Gram matrix of an real ETF for $\bbR^m$ of $n=v+1$ vectors.
Here, $\beta$ is necessarily \smash{$[\frac{n-m}{m(n-1)}]^{\frac12}$}.
\end{theorem}

To the best of our knowledge, the fact that $m$ is unique---that a single SRG cannot lead to multiple ETFs for spaces of various dimensions---has not appeared before in the literature.
Also, we did not find the above formula for $m$ in terms of $v$ and $k$ in the existing literature.
This is a minor realization, however, since this formula can obtained by inverting the change of variables given in Theorems~\ref{theorem.n-1 ETF to SRG}.
Though this can be seen by carefully following the details of the above discussion, we can also verify it directly:
\begin{lemma}
For any real numbers $m$, $n$, $v$, $k$ where $n>\max\set{m,1}$ and $v>0$,
\begin{equation*}
v=n-1,\ k=\tfrac n2-1+(\tfrac n{2m}-1)[\tfrac{m(n-1)}{n-m}]^{\frac12}
\quad
\Longleftrightarrow
\quad
n=v+1,\ m=\tfrac{v+1}{2}\bigset{1+\tfrac{v-2k-1}{[(v-2k-1)^2+4v]^{\frac12}}}.
\end{equation*}
\end{lemma}

\begin{proof}
($\Rightarrow$)
Given $m,n$ where $n>m$ and $n>1$, let $v=n-1>0$ and $k=\tfrac{n}{2}-1+\tfrac{\alpha-2}{2}(\tfrac{n-1}{\alpha-1})^{\frac12}=\tfrac{v-1}2+\tfrac{\alpha-2}{2}(\tfrac{v}{\alpha-1})^{\frac12}$ where $\alpha=\frac nm$.
Then we indeed have $v+1=n$ and also
\begin{equation*}
v-2k-1=v-(v-1)-(\alpha-2)(\tfrac{v}{\alpha-1})^{\frac12}-1=-(\alpha-2)(\tfrac{v}{\alpha-1})^{\frac12},
\end{equation*}
implying
\begin{equation*}
\tfrac{v+1}{2}\bigset{1+\tfrac{v-2k-1}{[(v-2k-1)^2+4v]^{\frac12}}}
=\tfrac{v+1}{2}\bigset{1-(\alpha-2)(\tfrac{v}{\alpha-1})^{\frac12}[\tfrac{\alpha-1}{(\alpha-2)^2v+4v(\alpha-1)}]^{\frac12}}
=\tfrac{v+1}{2}(1-\tfrac{\alpha-2}{\alpha})
=\tfrac{v+1}{\alpha}
=m.
\end{equation*}
($\Leftarrow$)
Given $v,k$ where $v>0$, let $n=v+1>1$ and $m=\tfrac{v+1}2(1+\varepsilon)$ where $\varepsilon=\tfrac{v-2k-1}{[(v-2k-1)^2+4v]^{\frac12}}$.
Since $\abs{\varepsilon}<1$ we have $n>m$.
Also, $n-1=v$ and
\begin{equation*}
\tfrac n2-1+(\tfrac n{2m}-1)[\tfrac{m(n-1)}{n-m}]^{\frac12}
=\tfrac{v+1}2-1+(\tfrac{1}{1+\varepsilon}-1)[\tfrac{(1+\varepsilon)v}{1-\varepsilon}]^{\tfrac12}\\
=\tfrac{v-1}2-\tfrac{\varepsilon v^{\frac12}}{1+\varepsilon}[\tfrac{(1+\varepsilon)^2}{1-\varepsilon^2}]^{\frac12}
=\tfrac{v-1}2-\tfrac{\varepsilon v^{\frac12}}{(1-\varepsilon^2)^{\frac12}}.
\end{equation*}
Since $\tfrac1{1-\varepsilon^2}=\tfrac{(v-2k-1)^2+4v}{(v-2k-1)^2+4v-(v-2k-1)^2}=\frac{(v-2k-1)^2+4v}{4v}$, this becomes
\begin{equation*}
\tfrac n2-1+(\tfrac n{2m}-1)[\tfrac{m(n-1)}{n-m}]^{\frac12}
=\tfrac{v-1}2-\tfrac{(v-2k-1)v^{\frac12}}{[(v-2k-1)^2+4v]^{\frac12}}\tfrac{[(v-2k-1)^2+4v]^{\frac12}}{2v^{\frac12}}
=\tfrac{(v-1)-(v-2k-1)}2
=k.\qedhere
\end{equation*}
\end{proof}

When taken together with the previous two theorems, this lemma implies that the ``if" statements of these theorems are actually ``if and only if."
For example, we have already seen that the parameters $(m,n)=(7,28)$ lead to the parameters $(v,k)=(27,16)$.
As such, Theorem~\ref{theorem.n-1 ETF to SRG} states that a $7\times 28$ real ETF implies the existence of a $\SRG(27,16,10,8)$.
Moreover, the previous lemma implies that under the change of variables of Theorem~\ref{theorem.n-1 SRG to ETF}, the parameters $(v,k)=(27,16)$ correspond to $(m,n)=(7,28)$, and so this result states that if there exists a $\SRG(27,16,10,8)$ then there exists a $7\times 28$ real ETF.
Altogether, we see that real ETFs are equivalent to SRGs in which $\mu=\frac k2$.

We conclude with a few minor observations.
Note that taking the graph complement of a given SRG equates to taking a Naimark complement of its corresponding real ETF and then negating $\varphi_1$.
As such, graph complements preserve our $\mu=\frac k2$ property.
This can also be easily seen by recalling the formulas for the parameters of the complement of an SRG:
\begin{equation*}
\tilde{\mu}
=v-2k+\lambda
=v-2k+\tfrac{3k-v-1}{2}
=\tfrac{v-k-1}{2}
=\tfrac{\tilde{k}}{2}.
\end{equation*}
Further note that taking a graph complement negates the deviation $\delta=v-2k-1$ of an SRG:
\begin{equation*}
\tilde{\delta}
=\tilde{v}-2\tilde{k}-1
=v-2(v-k-1)-1
=-(v-2k-1)
=-\delta.
\end{equation*}
This essentially corresponds to choosing ``$-$" in~\eqref{equation.traditional SRG 9} and~\eqref{equation.traditional SRG 11} instead of ``$+$".

\section*{Acknowledgments}

This work was partially supported by NSF DMS 1321779.
The views expressed in this article are those of the authors and do not reflect the official policy or position of the United States Air Force, Department of Defense, or the U.S.~Government.

\end{document}